\documentclass[10pt]{amsart}

\usepackage{amsfonts,amsmath,amsthm}
\usepackage{amssymb}
\usepackage{latexsym}
\usepackage{epsfig}
\usepackage{graphicx,color,graphics}

 \newtheorem{lemma}{Lemma}[section]
 
 \newtheorem{theorem}[lemma]{Theorem}
 
 \newtheorem{remark}[lemma]{Remark}

\newcommand{\N}{\ifmmode{{\Bbb N}}\else{\mbox{${\Bbb N}$}}\fi}
\newcommand{\R}{\ifmmode{{\Bbb R}}\else{\mbox{${\Bbb R}$}}\fi}

\newcommand{\cal}{\mathcal}

\title[Asymptotic behavior of a viscoelastic equation with second sound]{Asymptotic behavior of an inhomogeneous flexible structure with Cattaneo type of thermal effect}

\author{Margareth. S. Alves}
\address{Departamento de Matem{\'a}tica.
Universidade Federal de Vi{\c c}osa. UFV. MG. Brasil.} 
\email{{\tt malves@ufv.br}}

\author{Pedro Gamboa} 
\address{Instituto de
Matem{\'a}tica. UFRJ. Av. Athos da Silveira Ramos.
 RJ. Brazil}
 \email{{\tt pgamboa@im.ufrj.br}}
\author{Ganesh C. Gorain} 
\address{Department of Mathematics, Purulia
J.K.College. West Bengal. India.}
\email{{\tt goraing@gmail.com}}

\author{Amelie Rambaud}
\address{Universidad del B{\'i}o-B{\'i}o.
Concepci{\'o}n. Chile.} 
\email{{\tt arambaud@ubiobio.cl}} 

\author{Octavio Vera}
\address{Universidad del B{\'i}o-B{\'i}o.  Concepci{\'o}n. Chile.}
\email{{\tt overa@ubiobio.cl}}

\date{}

\begin{document}

\maketitle


%
%

\begin{abstract}
We consider vibrations of an
inhomogeneous flexible structure modeled by a $1$D viscoelastic equation with Kelvin-Voigt, coupled  with an expected dissipative effect : heat conduction governed by Cattaneo's law (second sound). We
establish the well-posedness of the system and we prove the stabilization to be exponential for one set of boundary conditions, and at least polynomial for another set of boundary conditions. Two different methods are used: the energy method and another more original, using the semigroup approach and studying the Resolvent of the system.  
\end{abstract}
\renewcommand{\theequation}{\thesection.\arabic{equation}}
\setcounter{equation}{0}
\section{Introduction and main results}\label{sec:intro}
One of the main issues concerning the vibrations in models of
flexible structural systems is the question of the stabilization of the 
structure. Indeed, one expects to prevent a system from resonance
effects, and wants to ensure a decay of the total
energy, at least polynomial, and hopefully exponential. It is
therefore of interest to investigate the theory behind the stabilization
processes in flexible structural systems and  to control their vibrations. One way to obtain a dissipative effect, and so a decay of the energy of the system, is to add a damping force. 
There exist various types of damping, such as boundary dampings,
internal dampings and localized dampings (see for example \cite{al1,liu1,chen1,ko1,la1,ma1} and references therein). 

In these kinds of problems, the best stability that one can expect is the so-called uniform stability. For example, in 2013, G. C. Gorain \cite{go1} has
established uniform exponential stability of the problem
\begin{eqnarray*}
\label{001}m(x)\,u_{tt} - (p(x)\,u_{x} + 2\,\delta(x)\,u_{xt})_{x} =
f(x),\quad\mbox{on}\quad (0,\,L)\times \mathbb{R}^{+},
\end{eqnarray*}
\noindent which describes the vibrations of an
inhomogeneous flexible structure with an exterior disturbing force
$f.$ More recently, M. Siddhartha {\it et. al.} \cite{misra1} showed the
exponential stability of the vibrations of a inhomogeneous flexible
structure with thermal effect governed by the Fourier law, 
\begin{eqnarray}
\label{002}& m(x)\,u_{tt} - (p(x)\,u_{x} + 2\,\delta(x)\,u_{xt})_{x}
- \kappa\,\theta_{x}= f, \\
\label{003}& \theta_{t} - \theta_{xx} - \kappa\,u_{xt} = 0.
\end{eqnarray}
\noindent Indeed, it is physically relevant to take into account thermal effects in flexible structures (see for example \cite{carlsonthermo}). However, in the above model, the temperature has an infinite velocity of propagation (heat equation): this property of the model is not consistent with the reality, where the heating or cooling of a flexible structure will usually take some time. Many researches have thus been conducted in order to modify the model of thermal effect. In the present paper, we will investigate a problem of vibrations for an inhomogeneous material of viscoelastic type (Kelvin-Voigt damping) subject to a thermal effect, now modeled by the so-called Cattaneo's law \cite{racke-second-sound-thermo}:
\begin{align}
\label{101}& m(x)\,u_{tt} - (p(x)\,u_{x} + 2\,\delta(x)\,u_{xt})_{x}
+ \eta\,\theta_{x}= 0, \\
\label{102}& \theta_{t} + \kappa\,q_{x} + \eta\,u_{xt} = 0, \\
\label{103}& \tau\,q_{t} + \beta\,q + \kappa\,\theta_{x}=0
\end{align}
\noindent where $x\in [0,\,L]$ and $t\geq 0$. Here, $\eta>0$ is the coupling
constant. $\beta,\ \kappa>0.$ $q=q(x,\,t)$ is the heat flux and the
parameter $\tau>0$ is the relaxation time describing the time lag in the response
for the temperature. Now the model of heat conduction is of hyperbolic type so that we have a finite speed of propagation. (Note that when taking formally $\tau \ = 0$ in the above system, we recover the viscothermoelastic system with the Fourier law \eqref{002}-\eqref{003}.)

 The functions $m(x),$ $\delta(x),$ and $p(x)$ are responsible for the inhomogeneous structure of the beam, and respectively denote mass per unit length of structure, coefficient
of internal material damping, and a positive function related to the
wave velocity at the vibrations at a point $x\in
\mathbb{R}^{*}=(0,\,+\infty).$  We will assume, in the rest of the paper:
\begin{equation}\label{hyp-p-delta-m}
m, \, \delta , \, p \, \in \, W^{1, \infty} \left(0, L\right) , \quad m (x), \, \delta (x) , \, p(x) \, > \, 0 ,\, \forall \, x \ \in \ [0, L]. 
\end{equation}
\noindent The initial conditions are given by
\begin{equation}\label{104}
u(x,\,0) = u_{0}(x),\, \, u_{t}(x,\,0)=u_{1}(x),\, \, 
\theta(x,\,0) = \theta_{0}(x),\, \,  q(x,\,0) = q_{0}(x).
\end{equation}
\noindent But concerning the boundary conditions, several choices are possible, depending on the physical situation one wants to deal with. Unfortunately, in general some lead to more tedious computations. Therefore, in the present paper we will deal with two sets of boundary conditions for system \eqref{101}-\eqref{103}. The first ones, corresponding to a rigidly clamped structure with temperature held constant at both extremities:
\begin{align}
\label{105}& u(0,\,t)=u(L,\,t) = 0,\qquad
\theta(0,\,t)=\theta(L,\,t) = 0,\qquad t\geq 0,
\end{align}
\noindent and the other one corresponding to a rigidly clamped structure with zero heat flux on the boundary:
\begin{align}
\label{bc-diri-q}& u(0,\,t)=u(L,\,t) = 0,\qquad
q(0,\,t)=q(L,\,t) = 0,\qquad t\geq 0.
\end{align}
For smooth solutions, this system enjoys  natural energy functionals $\mathcal{E}_1, \, \mathcal{E}_2 :\, \mathbb{R}^+ \ \rightarrow \ \mathbb{R}^+$, given by:
\begin{multline}\label{106} 
\mathcal{E}_{1}(t) =  \frac{1}{2}
\left[\int_{0}^{L}p(x)\,|u_{x}|^{2}\ dx + \int_{0}^{L}m(x)\,|u_{t}|^{2}\
dx \right.\\
 \left.+ \int_{0}^{L}|\theta|^{2}\ dx + \tau\int_{0}^{L}|q|^{2}\
dx\right],
 \end{multline}
\noindent and taking the time derivative of \eqref{101}-\eqref{104}, we build:
 \begin{multline}\label{108}
 \mathcal{E}_{2}(t) = \,  \frac{1}{2}
\left[\int_{0}^{L}p(x)\,|(u_{x})_{t}|^{2}\ dx +
\int_{0}^{L}m(x)\,|(u_{t})_{t}|^{2}\ dx \right. \\
\left. + \int_{0}^{L}|\theta_{t}|^{2}\ dx
+ \tau\int_{0}^{L}|q_{t}|^{2}\ dx\right]. 
\end{multline}
\noindent For smoother solutions we may generalize these energies up to the order $n \ \in \mathbb{N}$, as follows:
\begin{multline}\label{energy-n}
 \mathcal{E}_{n}(t) = \,  \frac{1}{2}
\left[\int_{0}^{L}p(x)\,|(u_{x})_{t\dots t}|^{2}\ dx +
\int_{0}^{L}m(x)\,|(u_{t})_{t\dots t}|^{2}\ dx \right. \\
\left. + \int_{0}^{L}|\theta_{t\dots t}|^{2}\ dx
+ \tau\int_{0}^{L}|q_{t\dots t}|^{2}\ dx\right],
\end{multline}
\noindent where $h_{t \dots t} = \frac{\partial^n}{\partial t^n} h$ for $h = u_x, \ u_t, \ \theta, $ or $q$. Moreover, it is straightforward to establish, for strong solution to \eqref{101}--\eqref{102} the dissipation of the energies , given in the following lemma. 
\begin{lemma}
\label{lemma101-102} For any strong solution to  system
\eqref{101}-\eqref{105} or \eqref{101}-\eqref{104}-\eqref{bc-diri-q}, smooth enough to define the energy functions \eqref{106},  \eqref{108}, or \eqref{energy-n}, then we have, for all $t > 0$:
\begin{eqnarray}
\label{107} \frac{d\mathcal{E}_{1}}{dt}(t) = -\
2\int_{0}^{L}\delta(x)\,|u_{xt}|^{2}\ dx - \beta\int_{0}^{L}|q|^{2}\
dx,
\end{eqnarray}
\begin{eqnarray}
\label{109}\frac{d\mathcal{E}_{2}}{dt}(t) = -\
2\int_{0}^{L}\delta(x)\,u_{xtt}^{2}\ dx -
\beta\int_{0}^{L}|q_{t}|^{2}\,dx,
\end{eqnarray}
\begin{eqnarray}
\label{diff-energy-n}\frac{d\mathcal{E}_{n}}{dt}(t) = -\
2\int_{0}^{L}\delta(x)\,|u_{xt\dots t}|^{2}\ dx -
\beta\int_{0}^{L}|q_{t\dots t}|^{2}\,dx.
\end{eqnarray}
\end{lemma}
The first energy estimate will allow us to investigate well-posedness with the point of view of semigroups \cite{pa1}. While the two last ones will be necessary to study the asymptotic behaviour. Actually, we expect the system to be exponentially stable, no matter the boundary conditions, but it appears that, depending on the boundary conditions chosen, the proof of such a result is more technical because of the second sound effect modeled by the Cattaneo law (see for example the discussion in \cite{racke-second-sound-thermo}). The main results of the present work are concerned with the  asymptotic behaviour of the system with either boundary conditions \eqref{105} or \eqref{bc-diri-q} and may be stated as follows.  
\begin{theorem}
\label{theorem3001}
 For any $n \in \mathbb{N} - \{0\}$, for suitable initial data (to be made explicit later, depending on $n$), the strong solution to system  \eqref{101}--\eqref{104} complemented by boundary conditions \eqref{105} satisfies, for all $t > 0$:
\begin{eqnarray}
\label{3000}{\cal E}_{n}(t) \leq \,\frac{[{\cal E}_{n}(0)
+ {\cal E}_{n+1}(0)]}{t},
\end{eqnarray}
that is to say the semigroup associated to the initial boundary value problem is (at least) polynomially stable, with a decay rate of $t^{-1}$.
\end{theorem}
The proof of this Theorem will use the energy method, and a suitable Lyapunov functional. 
\begin{theorem}
\label{theorem3000}
 For suitable initial data (to be made explicit later), the semigroup generated by system  \eqref{101}--\eqref{104} complemented by boundary conditions \eqref{bc-diri-q} is exponentially stable.
\end{theorem}
\noindent The proof will not use the second order energy, as it is generally done, but rather a semigroup point of view, with  a result due to  Pr\"{u}ss \cite{Pr84} and
Huang \cite{834231}:
\begin{theorem}{(Pr\"{u}ss)}\label{Pruss}
Let $({\cal S}(t))_{t\geqslant 0}$ be a $C_{0}$-semigroup
on a Hilbert space ${\cal H}$ generated by an operator ${\cal A}.$
The semigroup is exponentially stable if and only if
\begin{eqnarray*}
i\,\mathbb{R}\subset\varrho({\cal A}),\quad\mbox{and}\quad
\|(i\,\lambda\,I - {\cal A})^{-1}\|_{{\cal L}({\cal H})}
\leqslant C, \qquad\forall\,\lambda\in\mathbb{R}.
\end{eqnarray*}
\end{theorem}
\noindent Let us conclude the introduction by an important remark, related to the structure of system \eqref{101}-\eqref{104}. 
\begin{remark}\label{rem:mean-theta}
By formally integrating equation \eqref{102} over $(0,L)$, we get, for all $t > 0$:
$$ \frac{d}{d t}\ \int_{0}^L \ \theta(x,t) dx \ = \ \kappa (q(0,t) - q (L,t)) + \eta (u_t (0,t) - u_t (L,t)).$$
\noindent Therefore, we note that for boundary conditions \eqref{bc-diri-q}, the mean of $\theta$ is conserved in time, so that we may only study the problem for functions such that $\int_0^L \theta d x = 0$. Moreover, note that this can be required at least for $L^2$ functions since $(0,L)$ is bounded ($L^1 \subset L ^2$). This will be useful to investigate the exponential stability of the semigroup associated. 

\noindent However, for the other boundary conditions \eqref{105}, we would need observability estimates on the boundary terms for the unknown $q$ in order to control the term $\int |\theta |^2$. 
\end{remark}
\noindent The rest of the paper is organized as follows: Section \ref{sec:wellposedness} outlines briefly the notations and the well-posedness of the system is established with the semigroup approach. In Section \ref{sec:polin_decay}, we consider the boundary conditions \eqref{105} and show the polynomial stability of smooth solutions, using the energy method, and multiplier technique. Finally, in Section \ref{sec:expo_decay}, we show that for the boundary conditions \eqref{bc-diri-q}, the semigroup is exponentially stable, by studying the resolvent system.

\renewcommand{\theequation}{\thesection.\arabic{equation}}
\setcounter{equation}{0}
\section{Setting of the Semigroup}\label{sec:wellposedness}

In this section, we obtain existence and uniqueness of the
solution to the coupled system \eqref{101}-\eqref{104}, with either boundary conditions \eqref{105} or  \eqref{bc-diri-q},  using the semigroup approach.

\subsection{Notations}
Denote by $L^{2} (0,L)$ the classical set of $L^2$ functions over the interval $(0,L)$, equipped with the inner product and induced norm:
\begin{equation*}
\langle u,\,v\rangle _{L^{2}}=\int_{0}^{L}u\,\overline{v}
\,dx,\qquad \|u\|_{L^{2}}^{2}=\int_{0}^{L}|u|^{2}\ dx,
\end{equation*}
\noindent where we omit in the definition of the scalar product and norm the spatial space, here the interval $(0,L)$, for sake of clarity. Denote too by $H^1_0 (0,L)$ the Sobolev space of  homogeneous $H^1$ functions over $(0,L)$, equipped with its standard inner product. 
Let us now introduce the phase space 
$$\mathcal{H}=H_{0}^{1}(0,\,L)\times
L^{2}(0,\,L)\times L^{2}(0,\,L) \times L^{2}(0,\,L),.$$
\noindent  We define an inner product on $\mathcal{H}$: for $U^i =(u^i, w^i. \theta^i, q^i)$, $i = 1,2$, let
\begin{multline}\label{203}
\langle U^1,\,U^2\rangle_{\mathcal{H}}=
\int_{0}^{L}p(x)\,u^1_{x}\,\overline{u^2}_{x}\ dx +
\int_{0}^{L}m(x)\,w^1\,\overline{w^2}\ dx 
\\
+ \int_{0}^{L}\theta^1\,\overline{\theta^2}\ dx +
\tau\int_{0}^{L}q^1\,\overline{q^2}\ dx.
\end{multline}
\noindent Indeed, due to the hypothesis on $m$, $\delta$, $p$ \eqref{hyp-p-delta-m}, this provides an inner product on $\mathcal{H}$ and makes it a Hilbert space, equipped with the induced norm:
\begin{equation*}
\|U\|_{{\cal H}}^{2} =
\|\sqrt{p(x)}\,u_{x}\|_{L^{2}}^{2} +
\|\sqrt{m(x)}\,w\|_{L^{2}}^{2} +
\|\theta\|_{L^{2}}^{2} + \tau\,\|q\|_{L^{2}}^{2}.
\end{equation*}
\noindent Taking $u_{t}(x,\,t)=w(x,\,t),$ the initial boundary value problem can be reduced to the following abstract Cauchy problem for a first-order evolution equation
\begin{eqnarray}
\label{201}\frac{dU}{dt} = {\cal A}\,U, \qquad U(0)=U_{0}, \qquad
\forall \;t>0,
\end{eqnarray}
with the initial data
$\,U_{0}=(u_{0},\,w_{0},\,\theta_{0},\,q_{0})\in {\cal D}(\cal A),$ where the operator (formal up to now) 
 $\mathcal{A}:{\cal D}({\cal
A})\subseteq \mathcal{H}\rightarrow \mathcal{H}$ is given by
\begin{equation}
\label{202}\mathcal{A}\left(
\begin{array}{c}
u \\
w \\
\theta \\
q
\end{array}
\right) =\left(
\begin{array}{c}
w \\
\frac{1}{m(x)}\left(p(x)\,u_{x} + 2\,\delta(x)\,w_{x} - \eta\,\theta\right)_{x} \\
-\kappa\,q_{x} - \eta\,w_{x}\\
\frac{-1}{\tau}\left(\kappa\,\theta_{x} + \beta\,q\right)
\end{array}
\right).
\end{equation}
\noindent The domain of the operator, ${\cal D} ({\cal A})$, depends on the boundary conditions under consideration. 
For the boundary conditions \eqref{105}, we define:
\begin{multline}\label{204}
{\cal D}({\cal A}) \ = \ \mathcal{D}_1 \ = \  \big\{U=(u,\,w,\,\theta,\,q)\in {\cal H}:\
  w\in H_{0}^{1}(0,\,L),
  \\
  p(x)\,u_{x} + 2\,\delta(x)\,w_{x}\in
H^{1}(0,\,L), \ \theta\in H_0^{1}(0,\,L),\ q\in
H^{1}(0,\,L)\big\}.
\end{multline}
For the boundary conditions \eqref{bc-diri-q}, we define:
\begin{multline}\label{204-bc-diri-q}
{\cal D}({\cal A}) \ = \ \mathcal{D}_2 \ = \  \big\{U=(u,\,w,\,\theta,\,q)\in {\cal H}:\
  w\in H_{0}^{1}(0,\,L),
  \\
  p(x)\,u_{x} + 2\,\delta(x)\,w_{x}\in
H^{1}(0,\,L), \ \theta\in H^{1}(0,\,L),\ q\in
H_0^{1}(0,\,L)\big\}.
\end{multline}

We will now establish the well-posedness of the abstract Cauchy problem \eqref{201} thanks to the semigroup theory, in particular the Lummer-Phillips lemma (see for example \cite{pa1}).

\subsection{Well-posedness}
\begin{theorem}
\label{theo2.2} For any $\mathbf{U}_{0}\in\mathcal{H}$, there exists
a unique solution $\mathbf{U}(t)$ to the system
\eqref{101}-\eqref{104} with boundary conditions  \eqref{105} (resp. \eqref{bc-diri-q}), satisfying
\begin{equation*}
\mathbf{U}\in C([0,\,\infty[:\,\mathcal{H}).
\end{equation*}
If moreover, $\,\mathbf{U}_{0}\in{\cal D}_1$, given  by \eqref{204} (resp. $\mathcal{D}_2$, given by \eqref{204-bc-diri-q}), then
\begin{multline*}
 \mathbf{U}\in C^{1}([0,\,\infty[:\, \mathcal{H})\cap
C([0,\,\infty[:\,{\cal D}_1) \\
\Big( \hbox{ resp. } \, C^{1}([0,\,\infty[:\, \mathcal{H})\cap
C([0,\,\infty[:\, {\cal D}_2) \Big). 
\end{multline*}
\end{theorem}
\begin{proof}
 It suffices to show that the operator ${\cal A}$ is the generator infinitesimal of a $C_0$-semigroup of contractions on ${\cal H}$. Let us first show that ${\cal A}$ is dissipative. For $ U \in \mathcal{D} (A)$ (either $\mathcal{D}_1$ or $\mathcal{D}_2$), we compute:
\begin{eqnarray*}
\lefteqn{\left\langle\mathcal{A}U,\,U\right\rangle_{\mathcal{H}} =
\int_{0}^{L}p(x)\,w_{x}\,\overline{u}_{x}\ dx  +
\int_{0}^{L}(p(x)\,u_{x} + 2\,\delta(x)\,w_{x} -
\eta\,\theta)_{x}\,\overline{w}\ dx }
\\
& = & 2\,i\,Im\int_{0}^{L}p(x)\,w_{x}\,\overline{u}_{x}\ dx -
2\,i\,\eta\,Im\int_{0}^{L}\theta\,\overline{w}_{x}\ dx +
2\,i\,\kappa\,Im\int_{0}^{L}\overline{\theta}\,q_{x}\ dx     \\
&  & -\ 2\int_{0}^{L}\delta(x)\,|w_{x}|^{2}\ dx -
\beta\int_{0}^{L}|q|^{2}\ dx.
\end{eqnarray*}
Note that the same result is obtained whatever the boundary conditions under consideration. Taking the real part we obtain
\begin{eqnarray}
\label{205}Re\left\langle\mathcal{A}U,\,U\right\rangle_{\mathcal{H}}
= -\ 2\int_{0}^{L}\delta(x)\,|w_{x}|^{2}\ dx -
\beta\int_{0}^{L}|q|^{2}\leq 0 .
\end{eqnarray}
Thus the operator ${\cal A}$ is dissipative.  Next, $\mathcal{D}_i$, $i = 1, 2$ are obviously dense in $\mathcal{H}$ and  $\mathcal{A}$ is a closed operator.  It remains to show that $0\in\varrho({\cal A})$, the resolvent of the operator $\mathcal{A}$. Given $F=(f_{1},\,f_{2},\,f_{3},\,f_{4})\in {\cal H},$ we must show
that there exists a unique $U=(u,\,w,\,\theta,\,q)$ in ${\cal
D}({\cal A})$ such that ${\cal A}U=F,$ that is,
\begin{eqnarray}
\label{206}&  & w = f_{1}\quad\mbox{in}\quad H_{0}^{1}(0,\,L) \\
\label{207}&  & \left[p(x)\,u_{x} + 2\,\delta(x)\,w_{x} -
\eta\,\theta\right]_{x} = m(x)\,f_{2}
\quad\mbox{in}\quad L^{2}(0,\,L)  \\
\label{208}&  & \kappa\,q_{x} + \eta\,w_{x} = f_{3}
\quad\mbox{in}\quad L^{2}(0,\,L)  \\
\label{209}&  & \kappa\,\theta_{x} + \beta\,q = \tau\,f_{4}
\quad\mbox{in}\quad L^{2}(0,\,L).
\end{eqnarray}
\noindent We do the proof for the domain given by \eqref{204}, that is for boundary conditions \eqref{105}, since the other case can be done in a similar way, even easier.
First, replacing \eqref{206} into \eqref{208} we have
\begin{eqnarray}
\label{210}\kappa\,q_{x} = \eta\,f_{1x} + f_{3}\quad\mbox{in}\quad
L^{2}(0,\,L).
\end{eqnarray}
Therefore there is a unique $q\in H^{1}(0,\,L)$ satisfying
\eqref{210} given by
\begin{eqnarray}
\label{211}&  & \kappa\,q(x) =\kappa\,q(0) + \eta\,f_{1}(x) +
\int_{0}^{x}f_{3}(s)\ ds\quad\mbox{in}\quad [0,\,L]
\end{eqnarray}
where
\begin{eqnarray*}
q(0)=-\ \frac{\eta}{L}\int_{0}^{L}f_{1}(s)\ ds -
\frac{1}{L}\int_{0}^{L}\left(\int_{0}^{y}f_{3}(s)\ ds\right)dy -
\frac{\tau\,\kappa}{\beta\,L}\int_{0}^{L}f_{4}(s)\ ds.
\end{eqnarray*}
Moreover replacing \eqref{211} into \eqref{209} we have
\begin{eqnarray}
\label{212}&  & \kappa\,\theta_{x} =\beta \,q(0) + \frac{\beta\,
\eta}{\kappa}\;f_{1} + \frac{\beta}{\kappa}\int_{0}^{x}f_{3}(s)\ ds
+ \tau\,f_{4}
\end{eqnarray}
and it results that
\begin{eqnarray*}
\kappa\,\theta = \dfrac{\beta}{\kappa}\,q(0)\,x + \frac{\beta\,
\eta}{\kappa}\int_{0}^{x}f_{1}(s)\ ds +
\frac{\beta}{\kappa}\int_{0}^{x}\left(\int_{0}^{y}f_{3}(s)\
ds\right)dy + \tau\int_{0}^{x}f_{4}(s)\ ds
\end{eqnarray*}
belongs to $H_{0}^{1}(0,\,L)\cap H^{2}(0,\,L).$ On the other hand,
replacing \eqref{206} into \eqref{207} we have
\begin{eqnarray}
\label{213}-\ \eta\,\theta_{x} + (p(x)\,u_{x} +
2\,\delta(x)\,w_{x})_{x}= m(x)\,f_{2} \quad\mbox{in}\quad
H_{0}^{-1}(0,\,L).
\end{eqnarray}
Moreover, it is easy to verify that $\|U\|_{{\cal H}}\leq \|F\|_{{\cal H}}$.
Therefore $0\in \varrho({\cal A}).$ Then, applying the well known Lumer-Phillips theorem \cite{pa1}, ${\cal A}$ generates a semigroup of contraction and the proof of Theorem \ref{theo2.2} is achieved. 
\end{proof}

\renewcommand{\theequation}{\thesection.\arabic{equation}}
\setcounter{equation}{0}
\section{Asymptotic behaviour for the clamped structure with constant temperature on the boundary. }\label{sec:polin_decay}

With the notations of the previous section, we can reformulate precisely Theorem \ref{theorem3001} as follows. 
\begin{theorem}
\label{theorem3001-precise}
 For any $n \in \mathbb{N} - \{0\}$, let an initial datum $U_0 \ \in \mathcal{D} (\mathcal{A}^{n+1})$, the strong solution to system  \eqref{101}--\eqref{104} complemented by boundary conditions \eqref{105} satisfies, for all $t > 0$:
\begin{equation}\label{3000-n}
{\cal E}_{n}(t) \leq \,\frac{[{\cal E}_{n}(0)
+ {\cal E}_{n+1}(0)]}{t}.
\end{equation}
\end{theorem}
Note that we need to require more regularity on the initial datum than for the existence, in order to study the asymptotic behavior. A result with weaker hypothesis on the initial data is an on-going work. 

\begin{remark}
We expect actually to obtain a better result, that is an exponential decay. But up to now, we did not find the adequate Lyapunov function for the boundary conditions, and it is also an ongoing work. Indeed, the Kelvin-Voigt damping in the wave equation, as well as the fact that we consider a non homogeneous material, prevent us to find a Lyapunov function similar to the one proposed in \cite{racke-second-sound-thermo} for example. 
\end{remark}

Before proving Theorem \ref{theorem3001-precise}, we introduce some notations and classical Lemmas that we will need. 

\subsection{Notations and preliminary lemma}

\begin{lemma}
\label{lemma103} (Poincar\'{e} type Scheeffer's inequality, see
\cite{sche1})

Let $h \ \in H_0^1 (0,L)$. Then it holds, 
\begin{eqnarray}
\label{110}\int_{0}^{L}|h|^{2}\ dx \leq
\frac{L^{2}}{\pi^{2}}\int_{0}^{L}|h_{x}|^{2}\ dx.
\end{eqnarray}
\end{lemma}
\begin{lemma}
\label{lemma104} (Mean value theorem)
Let $(u, u_t, \theta, q)$ be the strong solution to \eqref{101}-\eqref{104}, with an initial datum in $\mathcal{D}(\mathcal{A})$. Then, for any $t > 0$, it exist a sequence of real numbers 
(depending on $t$), denoted by  $\xi_{i}$ $\in[0,\,L]$ $(i=1,\ldots,\,6)$ such that:
\begin{equation*}
\begin{array}{lr}
\int_{0}^{L}p(x)\,u_{x}^{2}\ dx =
p(\xi_{1})\int_{0}^{L}u_{x}^{2}\ dx,
& \int_{0}^{L}m(x)\,u^{2}\ dx =
m(\xi_{2})\int_{0}^{L}u^{2}\ dx,
\\
& \\
\int_{0}^{L}m(x)\,u_{t}^{2}\ dx =
m(\xi_{3})\int_{0}^{L}u_{t}^{2}\ dx,
& \int_{0}^{L}\delta(x)\,u^{2}\ dx =
\delta(\xi_{4})\int_{0}^{L}u^{2}\ dx,
\\
& \\
\int_{0}^{L}\delta(x)\,u_{x}^{2}\ dx =
\delta(\xi_{5})\int_{0}^{L}u_{x}^{2}\ dx,
& \int_{0}^{L}\delta(x)\,u_{xt}^{2}\ dx =
\delta(\xi_{6})\int_{0}^{L}u_{xt}^{2}\ dx.
\end{array}
\end{equation*}
\end{lemma}

\begin{proof}
 Since $m(x),$ $\delta(x),$ and $p(x)$ are continuous
function on $x\in [0,\,L]$, the conclusion is straightforward using the Mean Value Theorem. Moreover, it is obvious that $p(\xi_{1}),$ $m(\xi_{2}),$ $m(\xi_{3}),$
$\delta(\xi_{4}),$ $\delta(\xi_{5})$ and $\delta(\xi_{6})$ all are
positive and bounded from above and below.
\end{proof}

We will now define some auxiliary functionals that will help in the proof of Theorem \ref{theorem3001-precise}.  Let $(u, u_t, \theta, q)$ be the strong solution to \eqref{101}-\eqref{104}, with an initial datum in $\mathcal{D}(\mathcal{A}) = \mathcal{D}_1$ (given by \eqref{204}). We define
\begin{equation}\label{301}
{\cal F}_{1}(t) = \int_{0}^{L}m(x)\,u_{t}\,u\ dx +
\int_{0}^{L}\delta(x)\,u_{x}^{2}\ dx,
\end{equation}
and a Lyapunov functional
\begin{equation}\label{308}
 \mathcal{L}_1 = \mathcal{E}_1 + \mathcal{E}_2 + \varepsilon \mathcal{F}_1,
\end{equation}
\noindent where $\varepsilon$ is a non negative constant that will be adjusted later. 
%

Recalling the definitions of the first and second order energies \eqref{106} and \eqref{108}, we obtain: 
\begin{lemma}\label{lemma-estim-F1diff}
 Let $(u, u_t, \theta, q)$ be the strong solution to \eqref{101}-\eqref{105}, with an initial datum in $\mathcal{D}_1$. Then, for all $t > 0$,
\begin{equation}\label{estim-F1diff}
 \mathcal{F}_{1}' (t) \, = \, - 2 \, \mathcal{E}_1 (t) + \mathcal{R}_1 (t),
\end{equation}
\noindent where $\mathcal{R}_1$ is a remainder defined by:
$$ \mathcal{R}_1 (t) =  \int_0^L  \theta^2 + \tau \ \int_0^L  q^2 + 2 \int_{0}^L  m  u_t^2 - \eta \ \int_0^L  u \ \theta_x. $$
\end{lemma}

\begin{proof}
 Differentiating \eqref{301} in $t$, and using \eqref{101} \eqref{102} and the boundary conditions \eqref{105}, the result is straightforward.
\end{proof}
\noindent  We end this subsection by a lemma that gives an estimate from above and from below of the Lyapunov function $\mathcal{F}_1$ in terms of the energy $\mathcal{E}_1$.
\begin{lemma}\label{lemma-compar-F1-E1}
 Let $T > 0$, and Let $(u, u_t, \theta, q)$ be the strong solution to \eqref{101}-\eqref{105} on $(0,T)$, with an initial datum in $\mathcal{D}(\mathcal{A})$. Then, there exist two constants $\mu_0, \mu_1 > 0$, that depends only on the parameters of the problem, such that,  for all $t < T$,
 \begin{equation}\label{compar-F1-E1}
  - \mu_0 \ \mathcal{E}_1 (t) \ \leq \ \mathcal{F}_1 (t) \ \leq \ (\mu_0 + \mu_1)\  \mathcal{E}_1 (t).
 \end{equation}
\end{lemma}
\begin{proof}
 On the one hand, From the Young inequality, Lemma \ref{lemma104} and the definition of $\mathcal{E}_1$, we have for all $\alpha > 0$:
 $$\begin{array}{ll}
 \left\vert \int_0^L m u_t u \right\vert \ = \   \left\vert \int_0^L \big(\sqrt{m} u_t\big) \ \big(\sqrt{m} u\big) \right\vert
& \leq \ \alpha m (\xi_2) \ \int_0^L u^2 + \frac{1}{\alpha} \ \mathcal{E}_1 (t).
   \end{array}
 $$
 \noindent Next, applying the Poincar\'e Scheeffer type inequality \eqref{110}, and once again Lemma \ref{lemma104}, we get:
  $$\begin{array}{ll}
 \left\vert \int_0^L m u_t u \right\vert &  \leq \ \alpha \Vert m \Vert_{\infty} \frac{4 L^2}{\pi^2 \inf \vert p\vert } \ \int_0^L p u_x^2 + \frac{1}{\alpha} \ \mathcal{E}_1 (t).
   \end{array}
 $$
\noindent  Hence, since $\int_0^L p u_x^2 \ \leq \ \mathcal{E}_1$, we now choose $\alpha > 0$ such that 
$$ \alpha \Vert m \Vert_{\infty} \frac{4 L^2}{\pi^2 \inf \vert p\vert } = \frac{1}{ 2 \alpha},$$
\noindent namely $\alpha \ = \ \frac{\pi}{2 L} \ \sqrt{\frac{\inf \vert p \vert}{2 \Vert m \Vert_{\infty}}}$.  We thus define 
$$\mu_0 = \  \frac{2 L}{\pi} \ \sqrt{\frac{2 \Vert m \Vert_{\infty}}{\inf \vert p \vert}}. $$
\noindent This gives immediately the first (left) inequality of estimate \eqref{compar-F1-E1} since the other part of $\mathcal{F}_1$ is non negative. 
On the other hand, from Lemma \ref{lemma104} once again:
$$\begin{array}{ll}
  \int_0^L \delta  u_x^2  \ =   \ \delta (\xi_5) \int_0^L u_x^2
& \leq \ \mu_1 \ \mathcal{E}_1 (t),
   \end{array}
 $$
 \noindent with $\mu_1 =  \frac{\Vert \delta \Vert_{\infty}}{\inf |p|}$. And this concludes the proof of Lemma \ref{lemma-compar-F1-E1}.  
\end{proof}

We are now ready to prove the polynomial decay of the energy of our system with Dirichlet conditions for $\theta$. 

\subsection{Proof of Theorem \ref{theorem3001-precise} (Theorem \ref{theorem3001} in the introduction)}
We first prove the result for $n =1$. Let $U_0$ be an initial datum in $\mathcal{D} (\mathcal{A}^2)$, and $(u, u_t,\theta, q)$ the strong solution to system \eqref{101}-\eqref{104} with boundary conditions \eqref{105}.

\begin{lemma}\label{estim-F1diff-v2}
 \begin{equation}\label{esti-F1prime}
  \mathcal{F}_1' (t) \ \leq \  - C_1 \mathcal{E}_1 + C_2 \ \Big(\int_{0}^L  q^2  + \ \int_0^L q_t^2 + \ \int_0^L \delta u_{xt}^2\Big),
 \end{equation}
where $C_1, \, C_2 \ > \ 0$ will be made explicit in the proof.
\end{lemma}
\begin{proof}
From the equality \eqref{estim-F1diff} from Lemma \ref{lemma-estim-F1diff}, we have to estimate the remainder $\mathcal{R}_1$.  First, from the Poincar\'e estimate applied to $\theta$ (recall that we consider the boundary conditions \eqref{105}) and $u$, together with the Young inequality for the last term, we have, for all $\alpha > 0$:
\begin{multline}\label{estiR1}
 \mathcal{R}_1 \ \leq \ \left( \frac{L^2}{\pi^2} + \frac{\eta}{2 \alpha }\right) \ \int_0^L \theta_x^2 + \frac{\eta  L^2 \ \alpha}{2 \pi^2 \inf (p)} \ \int_0^L p u_x^2 
 \\
 + \frac{2 L^2 \vert m \vert_{\infty}}{\pi^2 \ \inf (\delta)} \ \int_0^L \delta u_{xt}^2 + \tau \ \int_0^L q^2. 
 \end{multline}

Chosing $\alpha >0$ small enough so that :
$$ C_1 : = \, 2  -   \frac{\eta  L^2 \ \alpha}{2 \pi^2 \inf (p)} \  >  0,$$
\noindent we absorb the term in $\int \ p \ u_x^2$ and get the first part of the inequality. Next, from equation \eqref{103} of our system, we get:
$$ \theta_x^2 \ = \ \frac{\tau^2}{\kappa^2} \ q_t^2  + \frac{2 \beta \tau}{\kappa^2} \ q \ q_t + \frac{\beta^2}{\kappa^2} q^2 .$$ 
Hence:
\begin{equation}
\label{305}\int_{0}^{L}\theta_{x}^{2} \  \leq
\frac{(\beta + \tau )^2}{\kappa^2}\left(\int_{0}^{L}q_{t}^{2} + \int_{0}^{L}q^{2}\
dx\right). 
\end{equation}
Therefore, injecting \eqref{305} into \eqref{estiR1} and combining with \eqref{estim-F1diff}, we get \eqref{esti-F1prime}, where $C_1$ has already been defined, while $C_2$ is given by:
$$C_2 \ = \ \max \left\{ \tau + \left( \frac{L^2}{\pi^2} + \frac{\eta}{2 \alpha }\right) \,\frac{(\beta + \tau )^2}{\kappa^2}\, ; \,  \left( \frac{L^2}{\pi^2} + \frac{\eta}{2 \alpha }\right) \, \frac{(\beta + \tau )^2}{\kappa^2} \, ; \, \frac{2 L^2 \vert m \vert_{\infty}}{\pi^2 \ \inf (\delta)} \, \right\} .$$
This ends the proof. Note that the parameter $\alpha > 0$ in $C_2$ is fixed. 
\end{proof}

Now, we are almost done. Coming back to our Lyapunov $\mathcal{L}_1$, differentiating with respect to time and using Lemma \ref{estim-F1diff-v2} and the energy equalities \eqref{107} and \eqref{109}:
\begin{multline}
 \frac{d}{d t} \mathcal{L}_1 \ \leq \ - \big(2 - \varepsilon \ C_2 \big) \ \int_0^L \delta u_{xt}^2 - 2 \int_0^L \delta u_{xtt}^2  
 \\
 - \big(\beta - \varepsilon \ C_2\big) \ \left( \int_{0}^{L}q_{t}^{2} + \int_{0}^{L}q^{2}\right) - \varepsilon \ C_1 \ \mathcal{E}_1. 
 \end{multline}
Hence, since $C_1$ and $C_2$ are already fixed, from the previous Lemma, we now choose $\varepsilon > 0$ so that:
$$ 2 - \varepsilon \ C_2  > 0 , \quad \beta - \varepsilon \ C_2  > 0. $$
It yields:
\begin{equation}\label{314} 
\frac{d}{d t} \mathcal{L}_1 \ \leq \  - \varepsilon \ C_1 \ \mathcal{E}_1. 
 \end{equation}
\noindent Now we choose $\varepsilon > 0$ such that, moreover: 
$$1 - \varepsilon \, \mu_0 \ \geq \ 0, $$
\noindent in order to ensure positivity of the Lyapunov $\mathcal{L}_1$ thanks to Lemma \ref{lemma-compar-F1-E1}. Finally, integrating \eqref{314} over $(0,\,t)$ and using that ${\cal E}_{1}$
is non increasing, we obtain
\begin{equation}
\label{315}t\,{\cal E}_{1}\leq\int_{0}^{L}{\cal E}_{1}(s)\ ds \leq
\frac{1}{\varepsilon \ C_1}\left({\cal L}(0) - {\cal L}(t)\right)\leq
\frac{{\cal L}(0)}{\varepsilon \ C_1}.
\end{equation}
Letting $C =1/(\varepsilon \ C_1) + \varepsilon \ \left( \mu_0 + \mu_1\right)$ (with the $\mu_i$ given by Lemma \ref{lemma-compar-F1-E1}) we have
\begin{eqnarray}
\label{316}{\cal E}_{1}(t) \leq \frac{C\,({\cal E}_{1}(0) +
{\cal E}_{2}(0))}{t},\qquad \forall\;t>0.
\end{eqnarray}
Now for $n \geq 2$, we define:
\begin{equation}\label{301-n}
{\cal F}_{n}(t) = \int_{0}^{L}m(x)\,u_{tt}\,u_t\ dx +
\int_{0}^{L}\delta(x)\,u_{xt}^{2}\ dx,
\end{equation}
and the Lyapunov functional
\begin{equation}\label{308-n}
 \mathcal{L}_n = \mathcal{E}_n + \mathcal{E}_{n+1} + \varepsilon \mathcal{F}_n,
\end{equation}
\noindent and proceed exactly as above. This ends the proof of  Theorem \ref{theorem3001-precise}.

\renewcommand{\theequation}{\thesection.\arabic{equation}}
\setcounter{equation}{0}
\section{Asymptotic behaviour for the clamped structure with zero flux on the boundary}\label{sec:expo_decay}

In this section, we will prove Theorem \ref{theorem3000}  given in Section \ref{sec:intro}. Precisely, we study the asymptotic behaviour of the solution to system \eqref{101}--\eqref{104} with boundary conditions \eqref{bc-diri-q}. 
\begin{theorem}
\label{theorem3000-precise}
 For initial data \eqref{104} within $\mathcal{D}_2$ (given by \eqref{204-bc-diri-q}), the semigroup generated by system  \eqref{101}--\eqref{104} complemented by boundary conditions \eqref{bc-diri-q} is exponentially stable.
\end{theorem}
We will prove this result thanks to  Theorem \ref{Pruss}. But recalling Remark \ref{rem:mean-theta}, and since the problem is linear, we can simplify the problem and assume that 
$$ \int_0^L \theta_0 \ = \ 0, $$
so that the temperature $\theta$ has zero mean value for every time. (if not, we have to consider the function $\hat{\theta}$). 
From now on, we thus suppose that for all $t \geq 0$,
$$\int_0^L \theta (t,x) \ d x \ = \ 0. $$ 
 Let us consider the resolvent system on the imaginary axis, for $F= (f_1, \dots,f_4) \in \mathcal{H}$, $\lambda \ \in \R$:
\begin{eqnarray}
 \label{resol1} i\lambda u - w & = f_1,
 \\
 \label{resol2} i \lambda m w - \left(p u_x + 2 \delta w_x \right)_x + \eta \theta_x & =  m f_2,
 \\
 \label{resol3} i \lambda \theta + \left(\kappa q + \eta w \right)_x & = f_3,
 \\
 \label{resol4} i \lambda\tau q + \kappa \theta_x + \beta q & = \tau f_4.
\end{eqnarray}
We will prove that the solution $U \in \mathcal{D}_2$ to this system (which exists, thanks to the previous section) satisfies: there exists a constant $C > 0$, independent of $U$ such that
$$ \Vert U \Vert_{\mathcal{H}} \ \leq \ C \ \Vert F \Vert_{\mathcal{H}}. $$
\noindent Theorem \ref{theorem3000} will then follow immediately from the characterization of exponentially stable semigroups given by Theorem \ref{Pruss}.

Let $U \in \mathcal{D}_2$. We first notice that, from the dissipativity of the operator $\mathcal{A}$, \eqref{205}, we have, taking the inner product of \eqref{resol1}-\eqref{resol4} together with $U$ and taking the real part:
$$2 \int_0^L \ \delta(x) |w_x|^2  + \beta \int_0^L \ |q|^2 \ = \ Re \big(\langle F, U \rangle \big),$$
so that we have two first estimates on the solution to the resolvent system, using Lemma \ref{lemma104}:
\begin{equation}\label{esti-resol1}
 \int_0^L \ m(x) |w_x|^2  + \tau  \int_0^L \ |q|^2 \leq \ C \Vert F \Vert_{\mathcal{H}} \ \Vert U \Vert_{\mathcal{H}}.
\end{equation}
\noindent Next, let us multiply \eqref{resol2} by $\overline{u}$, use \eqref{resol1} to eliminate $\lambda$ and integrate by parts. We obtain, since $u \in H^1_0 (0,L)$:
\begin{multline*}
 \int_{0}^L p (x) |u_x|^2 d x \ = \ - 2 \int_0^L \delta (x) w_x \overline{u}_x + \eta \int_0 ^L \theta  \overline{u}_x +\int_{0}^L |w|^2 
 \\ + \int_{0}^L \big( w \overline{f}_1 + \overline{u} m(x)f_2\big) dx.
\end{multline*}
Hence, using the Young inequality, the mean value lemma \ref{lemma104} and the Holder inequality, we get for $\alpha > 0$ small enough, there exists a constant $C_{\alpha} > 0$ such that:
\begin{equation*}
 \int_{0}^L \ p(x) |u_x|^2 \ \leq C_{\alpha} \Big( \int_0^L |\theta |^2 + \int_0^L | w_x|^2 + \Vert U \Vert_{\mathcal{H}} \Vert F \Vert_{\mathcal{H}}\Big).
\end{equation*}
 Hence, using the estimate \eqref{esti-resol1}, we get:
 \begin{equation}\label{esti-resol2}
 \int_0^L \ p(x) |u_x|^2  \leq \ C \Big( \Vert F \Vert_{\mathcal{H}} \ \Vert U \Vert_{\mathcal{H}} + \int_0^L |\theta |^2 \Big).
\end{equation}
\noindent Next, we multiply \eqref{resol4} by $\int_0^x \overline{\theta} (y) d y$ (which is well defined since in the domain, $\theta \ \in H^1 \ \subset \mathcal{C} (0,L)$), and use \eqref{resol3} to eliminate $\lambda$. It yields:
\begin{multline*}
 \kappa \int_0^L |\theta |^2 \ = \ \kappa \int_0^L | q |^2 + \eta \int_0^L q \overline{w} + \beta \int_{0}^L q \ \left(\int_{0}^x \overline{\theta} d y\right)
 \\
 + \Big[\kappa \theta (x) \left(\int_0^x \theta d y\right) \Big]_0^L - \int_0^L \ \Big(q \left(\int_0^x \overline{f}_3 dy \right) + \tau f_4 \left(\int_0^x \theta dy\right) \Big).
\end{multline*}
\noindent Now, since $\theta$ has zero mean value over $(0,L)$, we can eliminate the boundary terms appearing from the integrations by parts. Hence, by using again the Young inequality, together with Lemma \ref{lemma104} and Holder: for $\alpha > 0$ small enough, there exists $C_{\alpha} > 0$, such that
\begin{equation*}
 \kappa \int_0^L |\theta |^2 \ \leq \ C_{\alpha} \ \left(\int_0^L | q |^2 + \int_0^L |w|^2  + \Vert F \Vert_{\mathcal{H}} \ \Vert U \Vert_{\mathcal{H}} \right).
\end{equation*}
\noindent We conclude, thanks to the Poincar\'e estimates for $w \ \in H_0^1$  given by Lemma \ref{lemma103}, as well as the estimate \eqref{esti-resol1}, that:  
\begin{equation}\label{esti-resol3}
 \int_0^L |\theta |^2 \ \leq \ C \  \Vert F \Vert_{\mathcal{H}} \ \Vert U \Vert_{\mathcal{H}}. 
\end{equation}
Hence, combining \eqref{esti-resol1}, \eqref{esti-resol2} and \eqref{esti-resol3}, we get the wanted estimate and the proof of Theorem \ref{theorem3000-precise} is complete.

\section{conclusion}
In this study, we investigated the mathematical stability of the vibrations
of an inhomogeneous viscoelastic structure subject to a Cattaneo type law of heat conduction. We obtained exponential stability for Dirichlet conditions on the flux $q$ at the extremities, and polynomial stability when it is the temperature which satisfies Dirichlet conditions at the boundary. Indeed, these boundary conditions prevent us, up to now, to achieve exponential stability. However, we would expect the problem to be exponentially stable, no matter the boundary conditions, so that our result is a first step towards full stability analysis, even with mixed boundary conditions.  
\\

\paragraph{\emph{Acknowledgements}}
Octavio Vera thanks the support of the Fondecyt project 1121120. Amelie Rambaud thanks the support of the Fondecyt project 11130378. 



\end{document}